\documentclass[a4paper,10pt]{amsart}
\usepackage{enumerate}
\usepackage{amsthm,amsmath,amssymb,graphics,graphicx,hyperref,epstopdf,mathrsfs}
\usepackage{breqn}      
\usepackage{tikz}     
\title{Monomial multiplicities in explicit form}
\author{Guillermo Alesandroni}
\address{2000 Rosario, Santa Fe, Argentina}
\email{guillea@okstate.edu, alesandronig@yahoo.com}

\newtheorem{theorem}{Theorem}[section]
\newtheorem{proposition}[theorem]{Proposition}
\newtheorem{corollary}[theorem]{Corollary}
\newtheorem{lemma}[theorem]{Lemma}

\theoremstyle{definition}
\newtheorem{definition}[theorem]{Definition}
\newtheorem{remark}[theorem]{Remark}
\newtheorem{example}[theorem]{Example}
\newtheorem{construction}[theorem]{Construction}

\DeclareMathOperator{\betti}{b}
\DeclareMathOperator{\pd}{pd}
\DeclareMathOperator{\mdeg}{mdeg}
\DeclareMathOperator{\lcm}{lcm}

\DeclareMathOperator{\Max}{max}
\DeclareMathOperator{\reg}{reg}
\DeclareMathOperator{\hdeg}{hdeg}

\DeclareMathOperator{\pol}{pol}
\DeclareMathOperator{\codim}{codim}
\DeclareMathOperator{\degree}{d}
\DeclareMathOperator{\cdegree}{c}
\DeclareMathOperator{\e}{e}

\begin{document}
\maketitle
\begin{abstract}
Denote by $S$ a polynomial ring over a field, and let $M=(m_1,\ldots,m_q)$ be a monomial ideal of $S$. If $\codim(S/M) =1$, we prove that its multiplicity is given by
\[\e(S/M)= \deg(\gcd(m_1,\ldots,m_q)).\]
 On the other hand, if $M$ is a complete intersection, and $M'=(m_1,\ldots,m_q,m)$ is an almost complete intersection, we show that
 \[\e(S/M') = \prod\limits_{i=1}^q \deg(m_i) - \prod\limits_{i=1}^q \deg\left(\dfrac{m_i}{\gcd(m_i,m)}\right).\]
 We also introduce a new class of ideals that, broadly speaking, extends the family of monomial complete intersections and that of codimension 1 ideals, and give an explicit formula for their multiplicity. 

\end{abstract}

\section{Introduction}

Although the concept of multiplicity has been in our midst for a long time, an intensive research of this invariant was triggered by a series of conjectured multiplicity bounds introduced by Huneke, Herzog, and Srinivasan [HS, HS1] in the late 1990's and the early 2000's. As a result of the efforts of many mathematicians, some of these conjectures were proved in particular cases [FS], and finally, the general case was established using Boij-S\"oderberg theory [BS, EFW, ES].

With the multiplicity bounds still fresh in our minds, we now set the focus on a different target; finding exact values of the multipicities of classes of ideals. For instance, if $S$ represents a polynomial ring over a field, and $M=(m_1,\ldots,m_q)$ is a monomial ideal of $S$, we prove the following:

\begin{enumerate}[(i)]
\item  If $\codim(S/M) = 1$, then $\e(S/M)= \deg(\gcd(m_1,\ldots,m_q))$.
 \item  If $M$ is a complete intersection, and $M'=(m_1,\ldots,m_q,m)$ is an almost complete intersection, then $\e(S/M') = \prod\limits_{i=1}^q \deg(m_i) - \prod\limits_{i=1}^q \deg\left(\dfrac{m_i}{\gcd(m_i,m)}\right)$.
\end{enumerate}

Using monomial complete intersections as starting point and the formula for their multiplicities as inspiration, we construct a larger family of monomial ideals, and describe their multiplicities both algebraically and graphically.  

The organization of this article is as follows. In Section 2, we give the necessary background to understand this work. In Section 3, we express the multiplicities of ideals of codimension 1 in explicit form. In Section 4, we define the concept of stem ideal, which extends that of monomial complete intersection, and compute its multiplicity explicitly. In Section 5, we compute multiplicities in an even more general setting. In Section 6, we introduce a new approach, and use it to compute the multiplicities of monomial almost complete intersections.

\section{Background and notation}

Throughout this paper $S$ represents a polynomial ring in $n$ variables over a field. In some examples, $n$ takes a specific value, and  the variables are denoted with the letters $a$, $b$, $c$, etc. Everywhere else, $n$ is arbitrary, and $S$ is denoted $S=k[x_1,\ldots, x_n]$. The letter $M$ always represents a monomial ideal in $S$. 

We open this section by defining the Taylor resolution as a multigraded free resolution, something that will turn out to be fundamental in the present work. The construction that we give below can be found in [Me].
 
\begin{construction}
Let $M=(m_1,\ldots,m_q)$. For every subset $\{m_{i_1},\ldots,m_{i_s}\}$ of $\{m_1,\ldots,m_q\}$, with $1\leq i_1<\ldots<i_s\leq q$, 
we create a formal symbol $[m_{i_1},\ldots,m_{i_s}]$, called a \textbf{Taylor symbol}. The Taylor symbol associated to $\{\}$ will be denoted by $[\varnothing]$.
For each $s=0,\ldots,q$, set $F_s$ equal to the free $S$-module with basis $\{[m_{i_1},\ldots,m_{i_s}]:1\leq i_1<\ldots<i_s\leq q\}$ given by the 
${q\choose s}$ Taylor symbols corresponding to subsets of size $s$. That is, $F_s=\bigoplus\limits_{i_1<\ldots<i_s}S[m_{i_1},\ldots,m_{i_s}]$ 
(note that $F_0=S[\varnothing]$). Define
\[f_0:F_0\rightarrow S/M\]
\[s[\varnothing]\mapsto f_0(s[\varnothing])=s\]
For $s=1,\ldots,q$, let $f_s:F_s\rightarrow F_{s-1}$ be given by
\[f_s\left([m_{i_1},\ldots,m_{i_s}]\right)=
 \sum\limits_{j=1}^s\dfrac{(-1)^{j+1}\lcm(m_{i_1},\ldots,m_{i_s})}{\lcm(m_{i_1},\ldots,\widehat{m_{i_j}},\ldots,m_{i_s})}
 [m_{i_1},\ldots,\widehat{m_{i_j}},\ldots,m_{i_s}]\]
 and extended by linearity.
 The \textbf{Taylor resolution} $\mathbb{T}_M$ of $S/M$ is the exact sequence
 \[\mathbb{T}_M:0\rightarrow F_q\xrightarrow{f_q}F_{q-1}\rightarrow\cdots\rightarrow F_1\xrightarrow{f_1}F_0\xrightarrow{f_0} 
 S/M\rightarrow0.\]
 \end{construction}
We define the \textbf{multidegree} of a Taylor symbol $[m_{i_1},\ldots,m_{i_s}]$, denoted $\mdeg[m_{i_1},\ldots,m_{i_s}]$, as follows:
  $\mdeg[m_{i_1},\ldots,m_{i_s}]=\lcm(m_{i_1},\ldots,m_{i_s})$. 

\begin{definition}
Let $M$ be minimally generated by a set of monomials $G$.
\begin{itemize}
\item A monomial $m\in G$ is called \textbf{dominant} (in $G$) if there is a variable $x$, such that for all $m'\in G\setminus\{m\}$, the exponent with 
which $x$ appears in the factorization of $m$ is larger than the exponent with which $x$ appears in the factorization of $m'$.  In this case, we say that $m$ is dominant in $x$, and $x$ is a \textbf{dominant variable} for $m$.
\item $M$ is called a \textbf{dominant ideal} if each element of $G$ is dominant.
\end{itemize}
\end{definition}

 \begin{example}\label{example 1}
  Let $M_1$ and $M_2$ be minimally generated by $G_1=\{a^2,b^3,ab\}$ and $G_2=\{a^2b,ab^3c,bc^2\}$, respectively. Note that $a^2$ and $b^3$ are dominant in $G_1$, but $ab$ is not. Therefore, $M_1$ is not dominant. On the oher hand,  $a^2b$, $ab^3c$, and $bc^2$ are dominant in $G_2$ (note that $a$, $b$, and $c$ are dominant variables for $a^2b$, $ab^3c$, and $bc^2$, respectively). Thus, $M_2$ is a dominant ideal.
\end{example}

The next theorem gives a complete characterization of when the Taylor resolution is minimal [Al].

\begin{theorem}
With the above notation, $\mathbb{T}_M$ is minimal if and only if $M$ is dominant.
\end{theorem}

\begin{proof}
\end{proof}

The following classical result will be quoted often. Let $\mathbb{F}$ be a free resolution of $S/M$. If $[\sigma_{ij}]$ represents the $j$th basis element of $\mathbb{F}$ in homological degree $i$, define $\degree_{ij} = \deg(\mdeg[\sigma_{ij}])$. The Peskine-Szpiro formula [PS] states the following.

\begin{lemma} \label{Lemma 1}
With the above notation, we have

\[\sum\limits_{i} (-1)^i\sum\limits_{j} \degree_{ij}^k = \begin{cases}
															0 &\text{ for }1\leq k<c\\
															(-1)^c c! \e(S/M) &\text{ for }k=c,
															\end{cases}\]
where $c = \codim(S/M)$.
\end{lemma}

 \begin{remark} As pointed out in [HS1], The Peskine-Szpiro formula does not require  $\mathbb{F}$ to be minimal. In this article, the Peskine-Szpiro formula will be used in two particular cases; when $\mathbb{F}$=$\mathbb{T}_M$, and when $\mathbb{F}$ is a minimal resolution of $S/M$.
\end{remark}

\section{Multiplicity and codimension 1}

Let $m_1,\ldots,m_r$ be $r$ monomials of the form $m_i = x_1^{\alpha_{i1}}\ldots x_n^{\alpha_{in}}$. For each $i$, let $m_{i,\pol}$ denote the polarization of $m_i$; that is,
\[m_{i,\pol} = x_{11} \ldots x_{1 \alpha_{i1}} \ldots x_{n1} \ldots x_{n \alpha_{in}}.\]
Also, for each $i$, let $A_i = \{x_{11}, \ldots, x_{1 \alpha_{i1}}, \ldots, x_{n1}, \ldots, x_{n \alpha_{in}}\}$. The sets $A_i$ are said to be \textbf{associated} to the monomials $m_i$.  If $m_1,\ldots,m_r$ are the minimal generators of a monomial ideal $M$, then the Venn diagram displaying the sets $A_1,\ldots,A_r$ will be called the \textbf{diagram} of $M$ (Example \ref{Example stem-and-leaf} shows the diagram of an ideal $M$).
The next lemma states basic facts about sets and $\lcm$'s.

\begin{lemma} \label{Lemma 2}
Let $A_1, \ldots,A_r$ be the sets associated to $r$ monomials $m_1, \ldots,m_r$. Then
\begin{enumerate}[(i)]
\item $\deg(\lcm(m_1,\ldots,m_r)) = \# \left( \bigcup\limits_{i=1}^r A_i\right)$.
\item $\deg(\gcd(m_1,\ldots,m_r)) = \# \left(\bigcap\limits_{i=1}^r A_i\right)$.
\item $\deg \left( \dfrac{\lcm(m_1,\ldots,m_r)}{\lcm(m_2,\ldots,m_r)}\right) =\# \left( A_1 \setminus \bigcup\limits_{i=2}^r A_i\right)$.
\end{enumerate}
\end{lemma}

\begin{proof}
(i) Note that \[\lcm(m_1,\ldots,m_r) = \lcm(x_1^{\alpha_{11}} \ldots x_n^{\alpha_{1n}},\ldots, x_1^{\alpha_{r1}} \ldots x_n^{\alpha_{rn}}) = x_1^{\Max(\alpha_{11}, \ldots,\alpha_{r1})} \ldots x_n ^{\Max(\alpha_{1n},\ldots,\alpha_{rn})}\]
Hence,
\[\deg(\lcm(m_1,\ldots,m_r)) = \Max(\alpha_{11},\ldots,\alpha_{r1}) + \ldots + \Max(\alpha_{1n},\ldots,\alpha_{rn}).\]
Likewise, 
\begin{align*}
\bigcup\limits_{i=1}^r A_i & = \bigcup\limits_{i=1}^r \{x_{11},\ldots,x_{1 \alpha_{i1}},\ldots, x_{n1},\ldots, x_{n \alpha_{in}}\}\\
& = \bigcup\limits_{j=1}^n \left( \bigcup\limits_{i=1}^r \{x_{j1},\ldots,x_{j \alpha_{ij}}\}\right) \\
&= \bigcup\limits_{j=1}^n \{x_{j1},\ldots, x_{j\Max(\alpha_{1j},\ldots,\alpha_{rj})}\}\\
& =\left\{x_{11},\ldots,x_{1 \Max( \alpha_{11},\ldots,\alpha_{r1})} , \ldots , x_{n1},\ldots, x_{n \Max(\alpha_{1n},\ldots,\alpha_{rn})}\right\}.
\end{align*}
Therefore,
\[\# \left( \bigcup\limits_{i=1}^r A_i\right) = \Max(\alpha_{11},\ldots,\alpha_{r1}) + \ldots + \Max(\alpha_{1n},\ldots, \alpha_{rn}),\]
 and thus, $\deg(\lcm(m_1,\ldots,m_r)) = \# \left( \bigcup\limits_{i=1}^r A_i\right)$.\\
(ii) If in the proof of part (i) we change $\lcm$ for $\gcd$, $\Max$ for $\min$, and $\bigcup\limits_{i=1}^r$ for $\bigcap\limits_{i=1}^r$, we obtain that 
\[\deg(\gcd(m_1,\ldots,m_r)) = \#\left(\bigcap\limits_{i=1}^r A_i\right).\]
(iii) Finally, using part (i) we obtain 
\begin{align*}
\deg\left(\dfrac{\lcm(m_1,\ldots,m_r)}{\lcm(m_2,\ldots,m_r)}\right) & = \deg(\lcm(m_1,\ldots,m_r)) - \deg(\lcm(m_2,\ldots,m_r)) \\
&= \#\left(\bigcup\limits_{i=1}^r A_i\right) - \#\left(\bigcup\limits_{i=2}^r A_i\right) \\
&= \#\left(\bigcup\limits_{i=1}^r A_i \setminus \bigcup\limits_{i=2}^r A_i \right)\\
& = \# \left(A_1 \setminus \bigcup\limits_{i=2}^r A_i\right).
\end{align*}
\end{proof}

\textit{Notation}: Let $M=(m_1,\ldots,m_q)$. For all $i=1,\ldots,q-1$, let
 \[\mathscr{D}_i = \left\{\{m_{r_1},\ldots,m_{r_i}\}: 2\leq r_1<\ldots<r_i \leq q\right\}.\] Also, let $\mathscr{D}_0 = \mathscr{D}_q = \{\varnothing\}$. If $D=\{m_{r_1},\ldots,m_{r_i}\} \in \mathscr{D}_i$, then we make the following conventions:
\[\lcm(D) = \lcm(m_{r_1},\ldots,m_{r_i}); \quad \lcm(m_1,D) = \lcm(m_1,m_{r_1},\ldots,m_{r_i}).\]
In particular, if $D=\varnothing$, then $\lcm(D)=1$, and $\lcm(m_1,D) = \lcm(m_1) = m_1$.

\begin{theorem}\label{Theorem 3.2}
Let $M=(m_1,\ldots,m_q)$.
\begin{enumerate}[(i)]
\item  If $\codim(S/M) = 1$, then $\e(S/M)= \deg(\gcd(m_1,\ldots,m_q))$.
 \item  If $\codim(S/M) = q$, then $\e(S/M) = \prod\limits_{i=1}^q \deg(m_i)$.
\end{enumerate}
\end{theorem}

\begin{proof}
(i) By Lemma \ref{Lemma 1}, 
\begin{align*}
\e(S/M)& = \sum\limits_{i=1}^q (-1)^{i-1} \sum\limits_{j=1}^{q\choose i} \degree_{ij}\\
&= \sum\limits_{i=1}^q (-1)^{i-1} \sum_{\substack{[\sigma]\in \mathbb{T}_M \\ \hdeg[\sigma]=i}} \deg(\mdeg[\sigma])\\
& =\sum\limits_{i=1}^q (-1)^{i-1} \left[\sum\limits_{D\in \mathscr{D}_{i-1}} \deg(\lcm(m_1,D)) + \sum\limits_{D \in \mathscr{D}_i} \deg(\lcm(D))\right] \\
& = \sum\limits_{i=1}^ q (-1)^{i-1} \sum \limits_{D \in \mathscr{D}_{i-1}} \deg(\lcm(m_1,D)) - \sum\limits_{i=1}^q (-1)^i \sum\limits_{D \in \mathscr{D}_i} \deg(\lcm(D)) \\
&= \sum\limits_{i=1}^q (-1)^{i-1} \sum\limits_{D \in \mathscr{D}_{i-1}} \deg(\lcm(m_1,D)) - \sum\limits_{i=0}^{q-1} (-1)^i \sum\limits_{D \in \mathscr{D}_i} \deg(\lcm(D)) \\
&=\sum\limits_{i=1}^ q (-1)^{i-1} \sum\limits_{D \in \mathscr{D}_{i-1}} \deg(\lcm(m_1,D)) - \sum\limits_{i=1}^ q (-1)^{i-1} \sum\limits_{D \in \mathscr{D}_{i-1}} \deg(\lcm(D)) \\
& = \sum\limits_{i=1}^q (-1)^{i-1} \sum\limits_{D \in \mathscr{D}_{i-1}} \left[\deg(\lcm(m_1,D)) - \deg(\lcm(D))\right] \\
& = \sum\limits_{i=1}^q (-1)^{i-1} \sum\limits_{D\in \mathscr{D}_{i-1}} \deg \left(\dfrac{\lcm(m_1,D)}{\lcm(D)}\right) \\
&= \deg(m_1) + \sum\limits_{i=2}^q (-1)^{i-1} \sum\limits_{D \in \mathscr{D}_{i-1}} \deg \left(\dfrac{\lcm(m_1,D)}{\lcm(D)}\right) \\
&= \deg(m_1) + \sum\limits_{i=1}^{q-1} (-1)^ i \sum\limits_{D \in \mathscr{D}_i} \deg \left(\dfrac{\lcm(m_1,D)}{\lcm(D)}\right)\\
& = \deg (m_1) + \sum\limits_{i=1}^{q-1} (-1)^i \sum\limits_{2\leq r_1<\ldots < r_i \leq q} \deg\left(\dfrac{\lcm(m_1,m_{r_1},\ldots,m_{r_i})}{\lcm(m_{r_1},\ldots,m_{r_i})}\right).
\end{align*}

If we apply Lemma \ref{Lemma 2} (iii) to this last expression, we get
\begin{align*}
\e(S/M)& =\deg(m_1) + \sum\limits_{i=1}^{q-1} (-1)^i \sum\limits_{2\leq r_1<\ldots<r_i\leq q} \#\left(A_1\setminus \bigcup\limits_{k=1}^i A_{r_k} \right)\\
&= \deg(m_1) + \sum\limits_{i=1}^{q-1} (-1)^i \sum\limits_{2\leq r_1<\ldots<r_i\leq q} \#\left(\bigcap\limits_{k=1}^i (A_1\setminus A_{r_k})\right)\\
&= \deg(m_1) - \sum\limits_{i=1}^{q-1} (-1)^{i-1} \sum\limits_{2\leq r_1<\ldots<r_i\leq q} \#\left(\bigcap\limits_{k=1}^i (A_1\setminus A_{r_k})\right)
\end{align*}
Applying the principle of inclusion-exclusion, we get
\begin{align*}
\e(S/M) &= \deg(m_1) - \#\left(\bigcup\limits_{i=2}^q (A_1\setminus A_i)\right)\\
&= \deg(m_1) - \# \left(A_1 \setminus \bigcap\limits_{i=2}^q A_i\right) \\
&= \deg(m_1) - \# \left(A_1 \setminus \bigcap\limits_{i=1}^q A_i\right)\\
&= \deg(m_1) - \left[\# (A_1) - \# \left(\bigcap\limits_{i=1}^q A_i\right)\right].
\end{align*}
Finally, by Lemma \ref{Lemma 2} (ii),
\[\e(S/M) = \deg(m_1) - \left[\deg(m_1) - \deg(\gcd(m_1,\ldots,m_q))\right] = \deg(\gcd(m_1,\ldots,m_q)).\]
(ii) Since $M$ is minimally generated by $q$ monomials, and $\codim(S/M) = q$, no pair of minimal generators can be divided by a common variable. This implies that $M$ is a complete intersection, and the result holds.
\end{proof}

\begin{remark}: Theorem \ref{Theorem 3.2} (ii) simply paraphrases the well-known fact that the multiplicity of a complete intersection is the product of the degrees of its minimal generators. Including this result in Theorem \ref{Theorem 3.2} allows us to compare the multiplicities of monomial ideals in two opposite scenarios: when the codimension is minimal or maximal. The formulas given by Theorem \ref{Theorem 3.2} inspired the formula in Theorem \ref{Theorem stem-and-leaf}, where we give the multiplicities of some monomial ideals with intermediate codimension.
\end{remark}

\section{Multiplicity and stem ideals}

Suppose that $m_1$ is a dominant generator of $M = (m_1,\ldots,m_q)$. According to the third structural decomposition [Al1, Definition 4.7],
\[\betti_{k,l}(S/M) = \betti_{k,l}(S/M_1) + \betti_{k-1,\frac{l}{m_1}} (S/M_{m_1}),\]
where $M_1 = (m_2,\ldots,m_q)$, and $M_{m_1} = \left(\dfrac{\lcm(m_1,m_2)}{m_1},\ldots,\dfrac{\lcm(m_1,m_q)}{m_1}\right).$

\begin{lemma}\label{Lemma 3rd sd} 
Suppose that $m_1$ is a dominant generator of $M = (m_1,\ldots,m_q)$. Let $M_1$ and $M_{m_1}$ be the ideals determined by the third structural decomposition of $M$.
\begin{enumerate}[(i)]
\item If $\codim(S/M) = \codim(S/M_1) = \codim(S/M_{m_1})$, then 
\[\e(S/M) = \e(S/M_1) - \e(S/M_{m_1}).\]
\item If $\codim(S/M) = \codim(S/M_1) < \codim(S/M_{m_1})$, then
\[\e(S/M) = \e(S/M_1).\]
\end{enumerate}
\end{lemma}

\begin{proof}
Suppose that $\codim(S/M) = \codim(S/M_1) = c$. If we denote by $[\theta_{i,j}]$ (respectively, $[\sigma_{i,j}]$, and $[\tau_{i,j}]$) the $j^{th}$ basis element in homological degree $i$ of the minimal resolution of $S/M$ (respectively, $S/M_1$, and $S/M_{m_1}$) then, by the third structural decomposition,
\[\sum\limits_{j=1}^{\betti_i(S/M)} \left(\deg[\theta_{i,j}]\right)^c = \sum\limits_{j=1}^{\betti_i (S/M_1)} \left(\deg[\sigma_{i,j}]\right)^c + \sum\limits_{j=1}^{\betti_{i-1}(S/M_{m_1})} \left(\deg(m_1) + \deg[\tau_{i-1,j}]\right)^c\]
By the Peskine-Szpiro formula,

\begin{align*}
&\e(S/M) = \dfrac{1}{c!} \sum\limits_{i=1}^{\pd(S/M)} (-1)^{i+c}\sum\limits_{j=1}^{\betti_i(S/M)} (\degree_{ij})^c\\
&=  \dfrac{1}{c!} \sum\limits_{i=1}^{\pd(S/M)} (-1)^{i+c}\sum\limits_{j=1}^{\betti_i(S/M)} \left(\deg[\theta_{i,j}]\right)^c\\
&=  \dfrac{1}{c!} \sum\limits_{i=1}^{\pd(S/M)} (-1)^{i+c} \left[\sum\limits_{j=1}^{\betti_i(S/M_1)} \left(\deg[\sigma_{i,j}]\right)^c + \sum\limits_{j=1}^{\betti_{i-1}(S/M_{m_1})} \left(\deg(m_1) + \deg[\tau_{i-1,j}]\right)^c\right]\\
&= \dfrac{1}{c!} \sum\limits_{i=1}^{\pd(S/M_1)} (-1)^{i+c} \sum\limits_{j=1}^{\betti_i(S/M_1)} \left(\deg[\sigma_{i,j}]\right)^c + \dfrac{1}{c!} \sum\limits_{i=0}^{\pd(S/M)-1} (-1)^{i+1+c} \sum\limits_{j=1}^{\betti_i(S/M_{m_1})} \left(\deg(m_1) + \deg[\tau_{i,j}]\right)^c\\
&= \e(S/M_1) - \dfrac{1}{c!} \sum\limits_{i=0}^{\pd(S/M) - 1} (-1)^{i+c} \sum\limits_{j=1}^{\betti_i(S/M_{m_1})} \left(\deg(m_1) + \deg[\tau_{i,j}]\right)^c\\
&= \e(S/M_1) - \left[\dfrac{(-1)^c}{c!} \left(\deg(m_1)\right)^c + \sum\limits_{i=1}^{\pd(S/M_{m_1})} \dfrac{(-1)^{i+c}}{c!} \sum\limits_{j=1}^{\betti_i(S/M_{m_1})} \left(\deg(m_1) + \deg[\tau_{i,j}]\right)^c\right]\\
&= \e(S/M_1)- \dfrac{(-1)^c}{c!} \left(\deg(m_1)\right)^c - \sum\limits_{i=1}^{\pd(S/M_{m_1})} \dfrac{(-1)^{i+c}}{c!} \sum\limits_{j=1}^{\betti_i(S/M_{m_1})} \sum\limits_{k=0}^c {c\choose k}\left(\deg(m_1)\right)^{c-k} \left(\deg[\tau_{i,j}]\right)^k
\end{align*}
\begin{align*}
&=  \e(S/M_1)- \dfrac{(-1)^c}{c!} \left(\deg(m_1)\right)^c - \sum\limits_{k=0}^c {c\choose k}\left(\deg(m_1)\right)^{c-k} \left[\sum\limits_{i=1}^{\pd(S/M_{m_1})} \dfrac{(-1)^{i+c}}{c!} \sum\limits_{j=1}^{\betti_i(S/M_{m_1})} \left(\deg[\tau_{i,j}]\right)^k\right]\\
&= \e(S/M_1) - \dfrac{(-1)^c}{c!} (\deg(m_1))^ c + \dfrac{(-1)^c}{c!} (\deg(m_1))^c - \sum\limits_{i=1}^{\pd(S/M_{m_1})} \dfrac{(-1)^{i+c}}{c!} \sum\limits_{j=1}^{\betti_i(S/M_{m_1})} \left(\deg[\tau_{i,j}]\right)^c
\end{align*}

Therefore,\[\e(S/M) = \e(S/M_1) - \sum\limits_{i=1}^{\pd(S/M_{m_1})} \dfrac{(-1)^{i+c}}{c!} \sum\limits_{j=1}^{\betti_i(S/M_{m_1})} \left(\deg[\tau_{i,j}]\right)^c.\]
(i) Since $\codim(S/M_{m_1}) = c$,
\[\e(S/M_{m_1}) = \sum\limits_{i=1}^{\pd(S/M_{m_1})} \dfrac{(-1)^{i+c}}{c!} \sum\limits_{j=1}^{\betti_i(S/M_{m_1})} \left(\deg[\tau_{i,j}]\right)^c.\]
Hence,
\[\e(S/M) = \e(S/M_1) - \e(S/M_{m_1}).\]
(ii) Since $\codim(S/M_{m_1})\geq c+1$, 
\[\sum\limits_{i=1}^{\pd(S/M_{m_1})} \dfrac{(-1)^{i+c}}{c!} \sum\limits_{j=1}^{\betti_i(S/M_{m_1})} \left(\deg[\tau_{i,j}]\right)^c =0.\]
Hence,
\[\e(S/M) = \e(S/M_1).\]
\end{proof}

Next, we will introduce the class of stem ideals, and will give an explicit description of their multiplicities. The interesting fact about stem ideals is that they extend the class of dominant ideals of codimension 1, and that of complete intersections; and the formula for their multiplicities generalizes the statement of Theorem \ref{Theorem 3.2}.

\begin{definition} \label{Definition stem-and-leaf} 
A dominant ideal $M = (m_1,\ldots,m_q)$ will be called a \textbf{stem ideal} if, after reordering the minimal generators, there are integers $0=i_0<i_1<\ldots<i_c =q$, and monomials $l_1,\ldots,l_c$ of positive degree with the following properties:
\begin{enumerate}[(i)]
\item For each $1\leq t\leq c$, $\gcd(m_{i_{t-1}+1},m_{i_{t-1}+2},\ldots,m_{i_t}) = l_t$.
\item If $1\leq r<s\leq q$, and for some $t$, $r\leq i_t<s$, then $\gcd(m_r,m_s) = 1$.
\end{enumerate}
The monomials $l_{1},\ldots,l_{c}$ will be called \textbf{stems} of $M$. (Note that $\codim(S/M)=c$.)
\end{definition}

\begin{example} \label{Example stem-and-leaf} 
Let $M=(m_1=a^2bc, m_2=b^3c, m_3=c^4, m_4=d^2e^2, m_5=def, m_6=dg^2)$. We will show that $M$ is a stem ideal. Let $i_0=0$; $i_1=3$; $i_2=6$, and let $l_1=c; l_2=d$. Then
\begin{enumerate}[(i)]
\item $\gcd(m_1,m_2,m_3)=c=l_1$, and $\gcd(m_4,m_5,m_6) =d=l_2$.
\item Let $1\leq r \leq i_1=3< s \leq 6$. Since the only variables that appear in the factorization of $m_r$ are in  $\{a,b,c\}$, and the only variables that appear in the factorization of $m_s$ are in $\{d,e,f,g\}$, it follows that 
$\lcm(m_r,m_s) = 1$.
\end{enumerate} 
We have proven that $M$ is a stem ideal, with stems $l_1=c$, and $l_2=d$. The diagram of $M$ is\\

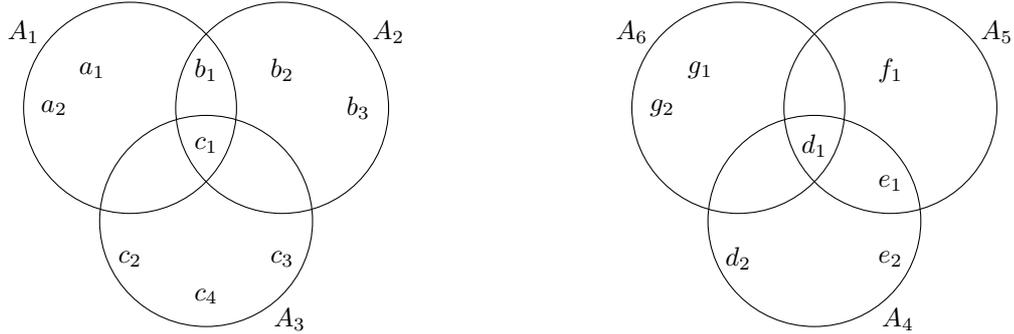
\begin{tikzpicture}
\draw (0,2) circle (1.4cm) ;
\draw (0,2.5)node{$b_2$};
\draw(1,2) node{$b_3$}; \draw (1.4,3) node {$A_2$};
\draw (-2,2) circle (1.4cm); \draw (-3.4,3) node {$A_1$};
\draw(-2.5,2.5) node{$a_1$};
\draw (-3,2) node {$a_2$};
\draw (-1,2.5) node {$b_1$};
\draw (-1,1.5) node {$c_1$};
\draw (-1,0.5) circle (1.4cm); \draw (0.1,-0.8) node {$A_3$};
\draw (-2,0) node {$c_2$};
\draw(0,0) node {$c_3$};
\draw (-1,-0.5) node {$c_4$};

\draw (8,2) circle (1.4cm) ;
\draw (8,2.5)node{$f_1$};
\draw(9.4,3) node{$A_5$};
\draw (6,2) circle (1.4cm); 
\draw(5.5,2.5) node{$g_1$};
\draw (5,2) node {$g_2$};
\draw (4.6,3) node {$A_6$};
\draw (7,1.5) node {$d_1$};
\draw (7,0.5) circle (1.4cm); \draw (8.1,-0.8) node {$A_4$};
\draw (6,0) node {$d_2$};
\draw(8,0) node {$e_2$};
\draw (8,1) node {$e_1$};

\end{tikzpicture}

Notice that the stems of $M$ can be read off of the diagram of $M$. All we need to do is identify the sets $A_1\cap A_2\cap A_3 = \{c_1\}$ and $A_4\cap A_5\cap A_6 = \{d_1\}$ with the monomials $\gcd(m_1,m_2,m_3)=c$ and $\gcd(m_4,m_5,m_6)=d$, respectively.
\end{example}

\begin{lemma} \label{Lemma stem-and-leaf} 
Let $M$ be a stem ideal. Suppose that, with the notation of Definition \ref{Definition stem-and-leaf}, $i_1-i_0\geq i_2-i_1\geq \ldots \geq i_c-i_{c-1} \geq 1$. Then
\begin{enumerate}[(a)]
\item If $i_1\geq 2$, and $\dfrac{\gcd(m_2,\ldots,m_{i_1})}{\gcd(m_1,m_2,\ldots,m_{i_1})} = 1$, then $\codim(S/M_{m_1}) \geq c+1$.
\item If $i_1\geq 2$, and $\dfrac{\gcd(m_2,\ldots,m_{i_1})}{\gcd(m_1,m_2,\ldots,m_{i_1})} \neq 1$, then $\codim(S/M_{m_1}) = c$.
\end{enumerate}
\end{lemma}

\begin{proof}
(a) When $i_1 = 2$, $\dfrac{\gcd(m_2,\ldots,m_{i_1})}{\gcd(m_1,m_2,\ldots,m_{i_1})}=\dfrac{\gcd(m_2)}{\gcd(m_1,m_2)} \neq 1$. Thus, we may assume $i_1 \geq 3$. By Lemma \ref{Lemma 2} (ii),

 $\#  \left( \bigcap\limits_{t=2}^{i_1} A_t \right)=\deg(\gcd(m_2,\ldots,m_{i_1})) = \deg(\gcd(m_1,\ldots,m_{i_1})) = \#  \left( \bigcap\limits_{t=1}^{i_1} A_t \right)$.
It follows that $\bigcap\limits_{t=1}^{i_1} A_t = \bigcap\limits_{t=2}^{i_1} A_t$. Hence, $\bigcap\limits_{t=2}^{i_1} A_t \subseteq A_1$ and $\bigcap\limits_{t=2}^{i_1} (A_t\setminus A_1) = \left[\bigcap\limits_{t=2}^{i_1} A_t\right] \setminus A_1 = \varnothing$.
For each $2\leq t\leq i_1$, let $m'_t = \dfrac{\lcm(m_1,m_t)}{m_1}$. Note that $A_t\setminus A_1$ is the set associated to $m'_t$. By Lemma \ref{Lemma 2} (ii),  $\deg(\gcd(m'_2,\ldots,m'_{i_1})) = \# \bigcap\limits_{t=2}^{i_1} (A_t\setminus A_1)= \#(\varnothing) = 0$. Therefore, $\gcd(m'_2,\ldots,m'_{i_1}) = 1$. It follows that $\codim\left(\dfrac{S}{(m'_2,\ldots,m'_{i_1})}\right) \geq 2$. On the other hand, $\codim\left(\dfrac{S}{(m_{i_1+1},\ldots,m_{i_c})}\right) = c-1$. Combining these facts, we get 
\[\codim\left(\dfrac{S}{M_{m_1}}\right) = \codim\left(\dfrac{S}{(m'_2,\ldots,m'_{i_1},m_{i_1+1},\ldots,m_{i_c})}\right) \geq 2 + (c-1 )= c+1.\]
(b) Since $\dfrac{\gcd(m_2,\ldots,m_{i_1})}{\gcd(m_1,m_2,\ldots,m_{i_1})} \neq 1$, we must have that $\gcd(m_2,\ldots,m_{i_1}) \nmid m_1$. Hence, $\bigcap\limits_{t=2}^{i_1} A_t \subsetneq A_1$. Thus, $\varnothing\subsetneq \left(\bigcap\limits_{t=2}^{i_1} A_t\right)\setminus A_1 = \bigcap\limits_{t=2}^{i_1} (A_t\setminus A_1)$. It follows that $1\leq \#\left(\bigcap\limits_{t=2}^{i_1} (A_t\setminus A_1)\right) = \deg \left(\gcd(m'_2,\ldots,m'_{i_1})\right)$. Hence, there is a variable dividing each of $m'_2,\ldots,m'_{i_1}$. Since $M_{m_1} = (m'_2,\ldots,m'_{i_1},m_{i_1+1},\ldots,m_{i_c})$, it follows that $\codim (S/M_{m_1}) = c.$
\end{proof}

The next theorem, which is the main result of this section, shows that the multiplicity of a stem ideal is the product of the degrees of its stems.

\begin{theorem}\label{Theorem stem-and-leaf} 
Let $M = (m_1,\ldots,m_q)$ be a stem ideal with stems $l_1,\ldots,l_c$. Then
\[\e(S/M) = \prod\limits_{t=1}^c \deg(l_{t})\]
\end{theorem}

\begin{proof}
Without loss of generality, we may assume that $i_1=i_1-i_0\geq i_2-i_1\geq \cdots \geq i_c-i_{c-1} \geq 1$. The proof is by induction on the number $q$ of minimal generators of $M$.\\
Suppose that $q=1$. Then $M=(m_1)$ is a complete intersection, and $\e(S/M) = \deg(m_1) = \deg(l_{1})$.\\
Suppose now that the theorem holds when $M$ is minimally generated by $q-1$ monomials. Let us prove the theorem in the case that $M$ is minimally generated by $q$ monomials.\\
We will consider three cases. \\
First case: $i_1=1$. In this case, we must have that $i_t = t$, for all $t=1,\ldots,c$. (In particular, $c=i_c=q$.) Thus, if $1\leq r<s\leq q$, we have that $r=i_r<s$ and, by property $(ii)$ of Definition \ref{Definition stem-and-leaf}, $\gcd(m_r,m_s) = 1$. This means that $M$ is a complete intersection and, therefore, $\e(S/M) = \prod\limits_{t=1}^q \deg(m_t) = \prod\limits_{t=1}^c \deg(m_{i_t}) = \prod\limits_{t=1}^c \deg(l_{t})$.\\
Second case: $i_1\geq 2$, and $\gcd(m_2,\ldots,m_{i_1}) = \gcd(m_1,m_2,\ldots,m_{i_1})$. Then $\codim(S/M) = \codim(S/M_1) = c$, and $\codim(S/M_{m_1})\geq c+1$, by Lemma  \ref{Lemma stem-and-leaf} (i). \\
It follows from Lemma \ref{Lemma 3rd sd} (ii), that $\e(S/M) = \e(S/M_1)$ and, by induction hypothesis, 
\begin{align*}
\e(S/M_1) &= \deg \left(\gcd(m_2,\ldots,m_{i_1})\right) \prod\limits_{t=2}^c \deg(l_{t})\\
& = \deg(\gcd(m_1,\ldots,m_{i_1})) \prod\limits_{t=2}^c \deg(l_{t})\\
& = \prod\limits_{t=1}^c \deg(l_{t}).
\end{align*}
Third case: $i_1\geq 2$, and $\gcd(m_2,\ldots,m_{i_1}) \neq \gcd(m_1,\ldots,m_{i_1})$. Then $\codim(S/M) = \codim(S/M_1) = \codim(S/M_{m_1}) = c$, and, by Lemma \ref{Lemma 3rd sd}(i), $\e(S/M) = \e(S/M_1) - \e(S/M_{m_1})$. Notice that

\begin{align*}
\deg\left(\gcd(m_2,\ldots,m_{i_1})\right) - \deg\left(\gcd(m'_2,\ldots,m'_{i_1})\right)
 &= \#\left[\bigcap\limits_{t=2}^{i_1} A_t\right] - \#\left[\bigcap\limits_{t=2}^{i_1} ((A_1\cup A_t)\setminus A_1)\right]\\
&=\#\left[\bigcap\limits_{t=2}^{i_1} A_t\right] - \#\left[\bigcap\limits_{t=2}^{i_1} (A_t\setminus A_1)\right]\\
& = \#\left[\bigcap\limits_{t=2}^{i_1} A_t\right] - \#\left[\left(\bigcap\limits_{t=2}^{i_1} A_t\right)\setminus A_1\right]\\
&= \#\left[\bigcap\limits_{t=1}^{i_1} A_t\right] \\
&= \deg\left(\gcd(m_1,\ldots,m_{i_1})\right).
\end{align*}
By induction hypothesis,
\begin{align*}
\e(S/M) &= \e(S/M_1) - \e(S/M_{m_1})\\
&= \deg\left(\gcd(m_2,\ldots,m_{i_1}) \prod\limits_{t=2}^c \deg(l_{t})\right) - \deg\left(\gcd(m'_2,\ldots,m'_{i_1}) \prod\limits_{t=2}^c \deg(l_{t})\right)\\
&= \left[\deg\left(\gcd(m_2,\ldots,m_{i_1})\right) - \deg\left(\gcd(m'_2,\ldots,m'_{i_1})\right)\right] \prod\limits_{t=2}^c \deg(l_{t})\\
&= \deg\left(\gcd(m_1,\ldots,m_{i_1})\right) \prod\limits_{t=2}^c \deg(l_{t})\\
&= \prod\limits_{t=1}^c \deg(l_{t}).
\end{align*}
\end{proof}

In Example \ref{Example stem-and-leaf}, the stems of a stem ideal $M$ are shown to be $l_1=c$ and $l_2=d$. By Theorem \ref{Theorem stem-and-leaf}, $\e(S/M) = \deg(l_{1})\deg(l_{2})=1.1=1$. In general, if $M = (m_1,\ldots,m_q)$ is a stem ideal, its multiplicity depends on its codimension, and is given by the following table:

\begin{tabular}{ c | c | c }
$\codim(S/M)$ & Diagram of $M$ & $\e(S/M)$\\
  \hline			
\begin{tikzpicture} \draw (1,2.7) node {$1$};  \draw (1,1) node {$$}; \end{tikzpicture}  & \begin{tikzpicture}
\draw (0,2) circle (0.8cm) ;
\draw (-0.1,1.5) circle (0.8cm);
\draw (-0.1,1.7) node {$l_1$};
\draw (-0.55,1.5) circle (0.8cm);
\draw (0.5,1.75) circle(0.8cm);
\draw (0,0.3) node {$q$ sets};
\draw (0,-0.1) node {$$};
\end{tikzpicture} & \begin{tikzpicture} \draw (1,2.7) node {$\deg(l_1)$};  \draw (1,1) node {$$}; \end{tikzpicture} \\
 \begin{tikzpicture} \draw (1,2.5) node {$2$};  \draw (1,1) node {$$}; \end{tikzpicture}  &\begin{tikzpicture}
\draw (0,2) circle (0.8cm) ;
\draw (-0.1,1.5) circle (0.8cm);
\draw (-0.1,1.7) node {$l_1$};
\draw (-0.55,1.5) circle (0.8cm);
\draw (0.5,1.75) circle(0.8cm);
\draw (0,0.3) node {$r$ sets};

\draw (3.3,2) circle (0.8cm) ;
\draw (3.2,1.5) circle (0.8cm);
\draw (3.2,1.7) node {$l_2$};
\draw (2.75,1.5) circle (0.8cm);
\draw (3.8,1.75) circle(0.8cm);
\draw (3.3,0.3) node {$q-r$ sets};
\end{tikzpicture}   & \begin{tikzpicture} \draw (1,2.5) node {$ \deg(l_1) . \deg(l_2)$};  \draw (1,1) node {$$}; \end{tikzpicture} \\
& &  \\
 $ \vdots$ &$ \vdots$ &$ \vdots$ \\
& &  \\
& &  \\
\begin{tikzpicture} \draw (1,1.7) node {$q$};  \draw (1,1) node {$$}; \end{tikzpicture} &\begin{tikzpicture}
\draw (-3,0) circle(0.4cm); \draw (-3,0) node {$l_1$}; \draw (-3,-0.7) node {$1$ set};
 \draw(-2,0) circle(0.4cm); \draw (-2,0) node{$l_2$}; \draw (-2,-0.7) node {$1$ set};
\draw (-1,0) node {$\ldots$};  \draw (-1.5,0) node {$$};  \draw(0,0) circle(0.4cm); \draw (0,0) node{$l_{q-1}$}; \draw (0,-0.7) node {$1$ set};
 \draw(1,0) circle(0.4cm); \draw(1,0) node {$l_q$}; \draw (1,-0.7) node {$1$ set};
\end{tikzpicture} & \begin{tikzpicture} \draw (1,1.7) node {$ \prod\limits_{i=1}^q \deg(l_i)$};  \draw (1,1) node {$$}; \end{tikzpicture} \\
  \hline  
\end{tabular} 

 \bigskip

\begin{remark}: In the diagrams above, abusing the notation, we have identified the sets associated to the stems $l_i$ with the stems themselves.
\end{remark}

\begin{remark}: In the codimension 1 case, we do not need to assume that $M$ is a stem ideal. The result follows from Theorem \ref{Theorem 3.2} (i), and holds for arbitrary monomial ideals. \\
\end{remark}
In the next three results we investigate quadratic dominant ideals, which are simple particular cases of stem ideals. We describe their multiplicity as well as their regularity, and show how one of these invariants can be expressed in terms of the other.

\begin{corollary} \label{coro}
Let $M$ be a quadratic dominant ideal with minimal generating set $G$. Then
\[\e(S/M) = 2^{\#U},\]
where $U =\left \{m\in G: \lcm(m,m') = 1, \text{ for all } m'\in G\setminus\{m\}\right\}$.
\end{corollary}

\begin{proof}
Let $M = (m_1,\ldots,m_q)$. Let $\{y_1,\ldots,y_h\} \subseteq \{x_1,\ldots,x_n\}$ be the set of all variables that divide more than one element of $G$. For all $t=1,\ldots,h$, let $U_t = \{m\in G: y_t\mid m\}$ (Notice that $\# U_t\geq 2$.) Without loss of generality, we may assume that there are integers $j_1<j_2<\ldots<j_h$, such that $U_1 = \{m_1,\ldots,m_{j_1}\}$, $U_2 = \{m_{j_1+1},m_{j_1+2},\ldots,m_{j_2}\}$, $\ldots$, $U_h = \{m_{j_{h-1}+1},m_{j_{h-1}+2},\ldots,m_{j_h}\}$. Then $U = \{m_{j_h+1},m_{j_h+2},\ldots,m_{j_h+v}=m_q\}$. We will prove that the integers $i_0=0,i_1=j_1,\ldots,i_h=j_h,i_{h+1}=j_h+1, i_{h+2}=j_h+2,\ldots,i_{h+v} = j_h+v=q$, and the monomials $l_t = \gcd(m_{i_{t-1}+1},\ldots,m_{i_t})$, with $1\leq t\leq h+v$, satisfy the two defining conditions of a stem ideal.\\
First property. If $1\leq t\leq h$, then $l_t = \gcd(U_t)$. Since each monomial of $U_t$ is divisible by $y_t$, $\deg(l_t) = \deg(\gcd(U_t))\geq 1$. That is, $l_t$ is a monomial of positive degree. On the other hand, if $t = h+w$, with $1\leq w\leq v$, then $l_t = m_{i_t}$ and, once again, $l_t$ is a monomial of positive degree.\\
Second property. Suppose that, for some $t$, $r\leq i_t<s$. If $s\geq i_{h+1} = j_h+1$, then $m_s\in U$ and, by definition, $\gcd(m_s,m_r) = 1$. On the other hand, if $s<i_{h+1}$, then for some $t<k\leq h$, $m_s\in U_k$, and $m_r\notin U_k$. Therefore, $y_k\mid m_s$ and $y_k\nmid m_r$. 
Let $m_s = y_k y$, with $y\in \{y_1,\ldots,y_h\}$. Since $\# U_k\geq 2$ and $m_s$ is a dominant monomial, $y$ must be a dominant variable of $m_s$, which implies that $y\nmid m_r$. Hence, $\gcd(m_s,m_r) = 1$. \\
We have proven that $M$ is a stem ideal with stems $l_1 = \gcd(U_1),\ldots,l_h = \gcd(U_h), l_{h+1} = m_{i_h+1},\ldots, l_{h+v} = m_{i_{h+v}}$. Since for all $t=1,\ldots,h$, $\#U_t\geq 2$, it follows that $\deg(l_t) = 1$. Finally, by Theorem\ref{Theorem stem-and-leaf}, 
\[\e(S/M) = \prod\limits_{t=1}^{h+v} \deg(l_t) = \prod\limits_{t=h+1}^{h+v} \deg(l_t) = \prod\limits_{t=h+1}^{h+v} \deg(m_{i_t}) = 2^v = 2^{\#U}.\]
\end{proof}

\begin{proposition} \label{prop}
Let $M = (m_1,\ldots,m_q)$ be a quadratic dominant ideal with minimal generating set $G$. Let $U =\left \{m\in G: \lcm(m,m') = 1, \text{ for all } m'\in G\setminus \{m\}\right\}$. Let $\{y_1,\ldots,y_k\} \subseteq \{x_1,\ldots,x_n\}$ be the set of all variables that divide more than one element of $G$. Then
 $\reg(S/M) = \#U+k = \codim(S/M)$.
\end{proposition}

\begin{proof}
For all $t=1,\ldots,k$, let $U_t = \{m\in G:y_t\mid m\}$. Let $i_t = \#U_t$. Then $U_t$ can be expressed in the form $U_t = \{y_tz_1,y_tz_2,\ldots,y_tz_{i_t}\}$. Hence, $\lcm(U_t) = y_tz_1z_2\ldots z_{i_t}$, and $\deg(\lcm(U_t)) = i_t+1$.\\
 Let $[\sigma] = [m_1,\ldots,m_q]$. Then $\deg(\mdeg[\sigma]) = \deg(\lcm(G)) = \sum\limits_{t=1}^k \deg(\lcm(U_t)) + \deg(\lcm(U)) =\sum\limits_{t=1}^k (i_t+1) + 2(\#U) = \sum\limits_{t=1}^k i_t + k+2 \#U$. Now, $q = \sum\limits_{t=1}^k \#U_t + \#U = \sum\limits_{t=1}^k i_t +\#U$. Thus, $\deg(\mdeg[\sigma]) = q+ [\#U+k]$. Hence, $\betti_{q,q+[\#U+k]}(S/M) \neq 0$, which means that $\reg(S/M)\geq \#U+k$.\\
Let $r = \reg(S/M) = \Max\{s: \betti_{i,i+s}(S/M) \neq 0\}$, and let $[\tau] = [h_1,\ldots,h_i]$ be a basis element of the minimal resolution of $S/M$, such that $\hdeg [\tau] = i$, and $\deg(\mdeg[\tau]) = i+r$. Since $M$ is dominant, if $h\in G\setminus \{h_1,\ldots,h_i\}$, then $h\nmid \lcm(h_1,\ldots,h_i)$. It follows that the basis element $[\tau'] = [h_1,\ldots,h_i,h]$ satisfies $\hdeg[\tau'] = i+1$; $\deg(\mdeg[\tau']) = i+r+j$, with $j\geq 1$. Hence, $\betti_{i+1,i+r+j}(S/M) \neq 0$. Since $\reg(S/M) = r$, we must have $j=1$. This implies that for all $k=i,\ldots,q$, there is a basis element $[\tau_k] \in \mathbb{T}_M$ that determines the regularity of $S/M$ (that is, $\hdeg[\tau_k] = k$, and $\deg(\mdeg[\tau_k]) = k+r$). In particular, for $[\tau_q] = [m_1,\ldots,m_q]$ we have $\hdeg[\tau_q] = q$, and $\deg(\mdeg[\tau_q]) = q+r$. Combining facts, we have that $\reg(S/M) = r = \#U+k$.

For each $m\in U$, let $x_m$ be a variable dividing $m$. It follows from the definition of $U$ that $x_m \neq x_m'$ if $m\neq m'$, and $\{x_m: m\in U \} \bigcap \{y_1,\ldots,y_k \}=\varnothing$. Let $X=\{x_m: m\in U \} \bigcup \{y_1,\ldots,y_k \}$. Note that every minimal generator is divisible by some variable of $X$. Thus, $\codim(S/M) \leq \#X=\#U+k$. Fot all $t=1,\ldots,k$, let $m'_t \in U_t$. Then the ideal generated by $\{m'_t: t=1,\ldots,k\} \bigcup U$ is a complete intersection. Thus, $\codim(S/M)=\#U+k$.
\end{proof}

\begin{corollary}
Let $M$ be a quadratic dominant ideal with minimal generating set $G$. Let $k$ be the number of variables that divide more than one element of $G$. Then
\[\e(S/M) = 2^{\reg(S/M) - k}= 2^{\codim(S/M) - k}\]
Equivalently, $\reg(S/M) =\codim(S/M)= k+ \log_2 \e(S/M)$.
\end{corollary}

\begin{proof}
It follows from Corollary \ref{coro} and Proposition \ref{prop}.
\end{proof}

 \section{Multiplicity and structural decomposition}

In this section we describe the multiplicities of ideals belonging to a class that extends the family of stem ideals.
 
Let $M$ be a dominant ideal such that $\codim(S/M)=c$. Suppose that $M$ can be expressed in the form $M=(m_1,\ldots,m_d,h_1,\ldots,h_c)$, where $(h_1,\ldots,h_c)$ is a complete intersection (note that stem ideals satisfy these hypotheses). Let $C=\{(j,\bar{m}) \in \mathbb{Z}^+\times S:$ there are integers $1\leq r_1<\ldots<r_j\leq d$, such that $\bar{m}=\lcm(m_{r_1},\ldots,m_{r_j})\} \cup \{(0,1)\}$. For each $(j,\bar{m}) \in C$, let $M_{\bar{m}}=(h'_1,\ldots,h'_c)$, where $h'_i=\dfrac{\lcm(\bar{m},h_i)}{\bar{m}}$ (in particular, $M_1=(h_1,\ldots,h_c)$). According to [Al1, Theorem 4.1], we have the following.
 
 \begin{theorem}\label{Theorem 5.1}
 For each integer $k$ and each monomial $l$,\\

 \[\betti_{k,l}(S/M)= \sum \limits_{(j,\bar{m}) \in C} \betti_{k-j,\frac{l}{\bar{m}}} (S/M_{\bar{m}}).\]
 \end{theorem}
 
  For practical reasons, the set $C$ defined above will be expressed in the form $C=\{(j_1,\bar{m}_1),\ldots, (j_w,\bar{m}_w)\}$. 
  The notation that we introduce now retains its meaning throughout this section. For each $i\geq 0$, let $\{[\theta_1^i],\ldots,[\theta_{t_i}^i]\}$ be the set of basis elements of the minimal resolution of $S/M$ in homological degree $i$ (notice that $t_i=\betti_i(S/M)$). Likewise, for each $i\geq 0$, and each $1\leq k\leq w$, let $\{[\sigma_{\bar{m}_k,1}^i],\ldots[\sigma_{\bar{m}_k,t_{\bar{m}_k,i}}^i]\}$ be the set of all basis elements of the minimal resolution of $S/M_{\bar{m}_k}$, in homological degree $i$ (thus, $t_{\bar{m}_k,i}=\betti_i(S/M_{\bar{m}_k})$).

 In order for the proof of the next theorem to be clear we will make the following convention. If $j_k, i$ are two integers such that $ i-j_k <0$, or $i-j_k \geq \pd(S/M_{\bar{m}_k})$, then we define $\deg[\sigma_{\bar{m}_k,j} ^{i-j_k}] + \deg(\bar{m}_k) = 0$. (The $[\sigma_{\bar{m}_k,j}^{i}]$ are the basis elements of the minimal resolution of $S/M_{\bar{m}_k}$ in homological degree $i$. In the proof of the next theorem, we will encounter expressions of the form $\deg[\sigma_{\bar{m}_k,j}^{i-j_k}] + \deg(\bar{m}_k)$. These expressions make sense when $0\leq i-j_k\leq \pd(S/M_{\bar{m}_k})$; otherwise, they do not, and we define them as $0$).

\begin{lemma} \label{Lemma 5.2}
With the above notation,
 \[\sum\limits_{i=1}^{\pd(S/M_{\bar{m}_k})} (-1)^i \sum\limits_{j=1}^{t_{\bar{m}_k,i}} \deg^l [\sigma_{\bar{m}_k,j}^i]=\begin{cases} 
  																						-1 &\text{ if } l=0\\0 &\text{ if } 1\leq l\leq c-1\\
																						 (-1)^c c! \e(S/M_{\bar{m}_k})&\text{ if } l=c.

																						\end{cases}\]
\end{lemma}
\begin{proof}
If $l=0$, then $\sum\limits_{i=1}^{\pd(S/M_{\bar{m}_k})} (-1)^i \sum\limits_{j=1}^{t_{\bar{m}_k,i}} \deg^l [\sigma_{\bar{m}_k,j}^i]=\sum\limits_{i=1}^{\pd(S/M_{\bar{m}_k})} (-1)^i \betti_i(S/M_{\bar{m}_k})=-1+\sum\limits_{i=0}^{\pd(S/M_{\bar{m}_k})} (-1)^i \betti_i(S/M_{\bar{m}_k})=-1.$\\
Let $1\leq l\leq c$. Recall that each $M_{\bar{m}_k}$ is of the form $M_{\bar{m}_k} = (h'_1,\ldots,h'_c)$, where 
  $h'_i=\dfrac{\lcm(m_{r_1},\ldots,m_{r_j},h_i)}{\lcm(m_{r_1},\ldots,m_{r_j})}$. Since $h'_i$ is a divisor of $h_i$, $M_{\bar{m}_k}$ is a complete intersection. Therefore, $\codim(S/M_{\bar{m}_k})=c$, and the formula above is simply a restatement of the Peskine-Szpiro formula.
\end{proof}

  \begin{theorem}\label{Theorem 2} 
  Let $M$ be a dominant ideal with $\codim(S/M)=c$. Suppose that $M$ can be expressed in the form $M=(m_1,\ldots,m_d,h_1,\ldots,h_c)$, where $(h_1,\ldots,h_c)$ is a complete intersection. Then, 
  \[\e(S/M)=\sum\limits_{k=1}^w (-1)^k \e(S/M_{\bar{m}_k}).\] 
  \end{theorem}
  
  \begin{proof}
  $(-1)^c c! \e(S/M)= \sum\limits_{i=1}^{\pd(S/M)} (-1)^i \sum\limits_{j=1}^{t_i} \deg^c [\theta_j^i]$
\begin{align*}
                           &= \sum\limits_{i=1}^{\pd(S/M)} (-1)^i \sum\limits_{k=1}^w \sum\limits_{j=1}^{t_{\bar{m}_k,(i-j_k)}} \left(\deg[\sigma_{\bar{m}_k,j}^{i-j_k}] + \deg(\bar{m}_k)\right)^c\\
                           &= \sum\limits_{k=1}^w \left\{ \sum\limits_{i=j_k}^{\pd(S/M)} (-1)^i \sum\limits_{j=1}^{t_{\bar{m}_k,(i-j_k)}} \left( \deg[\sigma_{\bar{m}_k,j}^{i-j_k}] + \deg(\bar{m}_k)\right)^ c\right\}\\
                           &= \sum\limits_{k=1}^w \left\{ \sum\limits_{i=0}^{\pd(S/M)-j_k} (-1)^{i+j_k} \sum\limits_{j=1}^{t_{\bar{m}_k,i}} \left( \deg[\sigma_{\bar{m}_k,j}^i] + \deg(\bar{m}_k)\right) ^c\right\}\\
                           &= \sum\limits_{k=1}^w \left\{ (-1)^{j_k} \deg^c(\bar{m}_k) + (-1)^{j_k}  \sum\limits_{i=1}^{\pd(S/M_{\bar{m}_k})} (-1)^i \sum\limits_{j=1}^{t_{\bar{m}_k,i}} \left( \deg[\sigma_{\bar{m}_k,j}^i] + \deg(\bar{m}_k)\right) ^c\right\}\\  
                           &=\sum\limits_{k=1}^w \left\{ (-1)^{j_k} \deg^c(\bar{m}_k) + (-1)^{j_k} \sum\limits_{i=1}^{\pd(S/M_{\bar{m}_k})} (-1)^i \sum\limits_{j=1}^{t_{\bar{m}_k,i}} \sum\limits_{l=0}^c {c\choose l} \deg^{c-l} (\bar{m}_k) \deg^l [\sigma_{\bar{m}_k,j}^i]\right\} \\
                           &= \sum\limits_{k=1}^w (-1)^{j_k} \deg^c(\bar{m}_k) + \sum\limits_{k=1}^w (-1)^{j_k} \sum\limits_{l=0}^c {c\choose l} \deg^{c-l}(\bar{m}_k) \sum\limits_{i=1}^{\pd(S/M_{\bar{m}_k})} (-1)^i 
                           \sum\limits_{j=1}^{t_{\bar{m}_k,i}} \deg^l [\sigma_{\bar{m}_k,j}].                      
  \end{align*}
Applying Lemma \ref{Lemma 5.2} to the last expression, we obtain
\[(-1)^c c! \e(S/M) =\sum\limits_{k=1}^w (-1)^{j_k} \deg^c(\bar{m}_k) + \sum\limits_{k=1}^w (-1)^{j_k} \deg^c (\bar{m}_k) (-1) + \sum\limits_{k=1}^w (-1)^{c+j_k} c! \e(S/M_{\bar{m}_k}).\]
 
Therefore,
\[(-1)^c c! \e(S/M)= \sum\limits_{k=1}^w (-1)^{c+j_k} c! \e(S/M_{\bar{m}_k}).\]
Finally, multiplying both sides by $\dfrac{(-1)^c}{c!}$, we obtain:
\[\e(S/M) = \sum\limits_{k=1}^w (-1)^{j_k} \e(S/M_{\bar{m}_k}).\]
  \end{proof}
  
  \begin{corollary}\label{Corollary 3}
  Let $M$ be a dominant ideal with $\codim(S/M)=c$. Suppose that $M$ can be expressed in the form $M=(m_1,\ldots,m_d,h_1,\ldots,h_c)$, where $(h_1,\ldots,h_c)$ is a complete intersection. Then,
  \[\e(S/M) = \sum\limits_{1\leq r_1<\ldots<r_j\leq d} (-1)^j \prod\limits_{i=1}^c \deg\left(\dfrac{\lcm(m_{r_1},\ldots,m_{r_j},h_i)}{\lcm(m_{r_1},\ldots,m_{r_j})}\right).\]
  \end{corollary}
  
  \begin{proof}
  According to Theorem \ref{Theorem 2}, $\e(S/M)=\sum\limits_{k=1}^w (-1)^{j_k} \e(S/M_{\bar{m}_k})$. By definition, each $M_{\bar{m}_k}$ is of the form $M_{\bar{m}_k} = (h'_1,\ldots,h'_c)$, where 
  $h'_i=\dfrac{\lcm(m_{r_1},\ldots,m_{r_j},h_i)}{\lcm(m_{r_1},\ldots,m_{r_j})}$ (with $j=j_k$). Since $h'_i$ is a divisor of $h_i$, $M_{\bar{m}_k}$ is a complete intersection and, thus, 
  \[\e(S/M_{\bar{m}_k}) = \prod\limits_{i=1}^c \deg(h'_i) = \prod\limits_{i=1}^c \deg\left(\dfrac{\lcm(m_{r_1},\ldots,m_{r_j},h_i)}{\lcm(m_{r_1},\ldots,m_{r_j})}\right).\]
  \end{proof}

A particular case where the hypotheses of Corollary \ref{Corollary 3} are satisfied is when $M$ is a dominant almost complete intersection. In such case, $M$ can be expressed in the form $M=(m, h_1,\ldots,h_c)$, where $(h_1,\ldots,h_c)$ is a complete intersection, and $\codim(S/M)=c$. By Corollary \ref{Corollary 3},
 \[\e(S/M) = \prod\limits_{i=1}^c \deg(h_i) - \prod\limits_{i=1}^c \deg\left(\dfrac{\lcm(m, h_i)}{m}\right).\]

We close this article with an example where we illustrate Theorem  \ref{Theorem 2} and Corollary \ref{Corollary 3}.
  
  \begin{example}
  Let $M=(a^3c,abe^3,a^2b^2,c^2,d^2e^2)$. Notice that $M$ is dominant, with $\codim(S/M)=3$. In addition, $(a^2b^2,c^2,d^2e^2)$ is a complete intersection. Also, $C=\{ (0,1); (1,a^3c); (1,abe^3); (2,a^3bce^3) \}$. By Theorem \ref{Theorem 2},\\

 \begin{align*}
\e(S/M)
&= (-1)^0 \e(S/M_{1})+ (-1)^1 \e(S/M_{a^3c})+  (-1)^1 \e(S/M_{abe^3})+  (-1)^2 \e(S/M_{a^3bce^3})\\
&=(-1)^0 \e(S/(a^2b^2,c^2,d^2e^2))+ (-1)^1 \e(S/(b^2,c,d^2e^2))+  (-1)^1 \e(S/(ab,c^2,d^2))\\
&+  (-1)^2 \e(S/(b,c,d^2)).
 \end{align*}

Hence,\\
  \begin{align*}
  \e(S/M)
   &= (-1)^0 [\deg(a^2b^2) \deg(c^2) \deg(d^2e^2)] + (-1)^1 [\deg(b^2)\deg(c) \deg(d^2e^2)] \\
  &+ (-1)^1 [\deg(ab) \deg(c^2) \deg(d^2)] 
  + (-1)^2 [\deg(b) \deg(c) \deg(d^2)] \\
  &= 4.2.4-2.1.4-2.2.2+1.1.2 = 32-8-8+2 = 18.
  \end{align*}
  \end{example}

\section{A different approach}

In this section we introdude a new line of reasoning. The idea is to express the multiplicity of an ideal $M$ in terms of the multiplicity of some other ideal $M'$, and take full advantage of the Peskine-Szpiro formula. In particular, we will use the sometimes overlooked fact that the formula in Lemma \ref{Lemma 1} equals o when $k$ is less than the codimension of the ideal.

Below we sketch a new nad easier proof of Theorem \ref{Theorem 3.2}(i), where we illustrate this idea.

Let $M = (m_1,\ldots, m_q)$, and suppose that $\codim(S/M) = 1$. Set $l = \gcd(m_1,\ldots,m_q)$, and let $M'=(m'_1,\ldots,m'_q)$, where $m'_i = \dfrac{m_i}{l}$. Let $[\sigma_{ij}]$ (respectively, $[\tau_{ij}]$) be the $j$th basis element of $\mathbb{T}_M$ (respectively, $\mathbb{T}_{M'}$) in homological degree $i$, and let $\degree_{ij} = \deg(\mdeg[\sigma_{ij}])$ (respectively, $\cdegree_{ij} = \deg(\mdeg[\tau_{ij}])$). Since $\codim(S/M')\geq 2$, the formula of Lemma \ref{Lemma 1} yields
\[\sum\limits_{i=1}^q (-1)^i \sum\limits_{j=1}^{q\choose i} \cdegree_{ij} = 0\] Since $\codim(S/M) = 1$, we have

\begin{align*}
(-1)^1 1!\e(S/M) & =\sum\limits_{i=1}^q (-1)^i \sum\limits_{j=1}^{q\choose i} \degree_{ij}\\
& = \sum\limits_{i=1}^q (-1)^i \sum\limits_{j=1}^{q\choose i} [\deg(l)+ \cdegree_{ij}]\\
&= \sum\limits_{i=1}^q (-1)^i \sum\limits_{j=1}^{q\choose i} \deg(l) + \sum\limits_{i=1}^q (-1)^i \sum\limits_{j=1}^{q\choose i} \cdegree_{ij}\\
&= \deg(l) \sum\limits_{i=1}^q (-1)^i {q\choose i}\\
&= -\deg(l)
\end{align*}
Hence, $\e(S/M) = \deg(l) = \deg( \gcd(m_1,\ldots,m_q))$.

The next theorem, which gives the multiplicities of all monomial almost complete intersections, extends the result following Corollary \ref{Corollary 3} and, once more, illustrates the new approach introduced at the begining of this section. First we need a lemma.

\begin{lemma}\label{Lemma alpha}
Let $M = (m_1,\ldots,m_q,m)$ be an almost complete intersection, such that $M$ is not dominant, and $M_1 = (m_1,\ldots,m_q)$ is a complete intersection. Then, for some $1\leq i\leq q$, $M' = (m_1,\ldots,\widehat{m_i},\ldots,m_q,m)$ is a dominant almost complete intersection, where $M'' = (m_1,\ldots,\widehat{m_i},\ldots,m_q)$ is a complete intersection.
\end{lemma}

 \begin{proof}
Let $1\leq r\leq q$. Since $m_r \nmid m$, there is a variable $x$ such that the exponent with which $x$ appears in the factorization of $m_r$ is larger than the exponent with which $x$ appears in the factorization of $m$. In addition, none of the monomials $m_1,\ldots,\widehat{m_{r}},\ldots,m_q$ is divisible by $x$, for $M_1$ is a complete intersection. Hence, $m_r$ is a dominant generator of $M$. Since $r$ is arbitrary, we have that $m_1,\ldots,m_q$ are dominant generators of $M$ and given that $M$ is not dominant, $m$ must be a nondominant generator of $M$. 

Let $y$ be a variable dividing $m$. Since $m$ is nondominant, there is a generator $m_i$ such that $y \mid m_i$. Given that $m \nmid m_i$ there is a variable $z$ such that the exponent with which $z$ appears in the factorization of $m$ is larger than the exponent with which $z$ appears in the factorization of $m_i$. Once again, since $m$ is not dominant, there is a generator $m_j$ (with $i \neq j$)  such that $z \mid m_j$. 

Therefore, the ideal $M' = (m_1,\ldots,\widehat{m_i},\ldots,m_q,m)$ is dominant. Also, since $M'' = (m_1,\ldots,\widehat{m_i},\ldots,m_q)$ is a complete intersection and $\gcd (m_j,m) \neq 1$, it follows that $M' = (m_1,\ldots,\widehat{m_i},\ldots,m_q,m)$ is an almost complete intersection. 
 \end{proof}

 \begin{theorem}\label{ACI}
Let  $M = (m_1,\ldots,m_q,m)$ be an almost complete intersection, where $M_1= (m_1,\ldots,m_q)$ is a complete intersection. Then
 \[\e(S/M) = \prod\limits_{i=1}^q \deg(m_i) - \prod\limits_{i=1}^q \deg\left(\dfrac{m_i}{\gcd(m_i,m)}\right).\]
\end{theorem}
 
\begin{proof}
If $M$ is dominant, the theorem follows from Corollary \ref{Corollary 3}. Thus, we may assume that $M$ is not dominant. By Lemma \ref{Lemma alpha}, for some $1\leq i\leq q$, the ideal $M' = (m_1,\ldots,\widehat{m_i},\ldots,m_q,m)$ is a dominant almost complete intersection, and the ideal $M'' = (m_1,\ldots,\widehat{m_i},\ldots,m_q)$ is a complete intersection. Without loss of generality, we may assume that $M_2=(m_1,\ldots,m_{q-1})$ is a complete intersection, and $M_3=(m_1,\ldots,m_{q-1},m)$ is a dominant almost complete intersection.

Denote by $[\theta_{ij}]$ (respectively, $[\tau_{ij}])$ the $j$th basis element of $\mathbb{T}_M$ (respectively, $\mathbb{T}_{M_2}$) in homological degree $i$. For $[\tau_{ij}]=[m_{r_1},\ldots,m_{r_i}]$, define  $[\tau_{ij},m_q]=[m_{r_1},\ldots,m_{r_i},m_q]$,  $[\tau_{ij},m]=[m_{r_1},\ldots,m_{r_i},m]$, and  $[\tau_{ij},m_q,m]=[m_{r_1},\ldots,m_{r_i},m_q,m]$.

Since $\codim(S/M) = q$, we have that
\begin{align*}
&(-1)^q q! \e(S/M)=\sum\limits_{i=1}^{q+1} (-1)^i \sum\limits_{j=1}^{{q+1}\choose i} \deg^q[\theta_{ij}]\\
&= \sum\limits_{i=0}^{q-1} (-1)^i \sum\limits_{j=1}^{{q-1}\choose i} \left(\deg^q[\tau_{ij}]- \deg^q[\tau_{ij},m_q]\right) + \sum\limits_{i=0}^{q-1} (-1)^i \sum\limits_{j=1}^{{q-1}\choose i} \left(\deg^q[\tau_{ij},m_q,m]- \deg^q[\tau_{ij},m]\right)\\
&=(-1)^q q! \e(S/M_1)+\sum\limits_{i=0}^{q-1} (-1)^i \sum\limits_{j=1}^{{q-1}\choose i} \left(\deg^q[\tau_{ij},m_q,m]- \deg^q[\tau_{ij},m]\right)\\
&=(-1)^q q!  \prod\limits_{i=1}^q \deg(m_i) + \sum\limits_{i=0}^{q-1} (-1)^i \sum\limits_{j=1}^{{q-1}\choose i} \left [\left(\deg[\tau_{ij},m]+\deg\left(\dfrac{m_q}{\gcd(m_q,m)}\right)\right)^q- \deg^q[\tau_{ij},m]\right]\\
& = (-1)^q q!  \prod\limits_{i=1}^q \deg(m_i)+ \\
&+ \sum\limits_{i=0}^{q-1} (-1)^i \sum\limits_{j=1}^{{q-1}\choose i} \left[ \sum\limits_{k=0}^q {q\choose k} \deg^k[\tau_{ij},m] \deg^{q-k}\left(\dfrac{m_q}{\gcd(m_q,m)}\right) - \deg^q[\tau_{ij},m]\right]\\
&=(-1)^q q!  \prod\limits_{i=1}^q \deg(m_i)+\\
&+  \sum\limits_{i=0}^{q-1} (-1)^i \sum\limits_{j=1}^{{q-1}\choose i} \left[ \sum\limits_{k=0}^{q-1} {q\choose k} \deg^k[\tau_{ij},m] \deg^{q-k}\left(\dfrac{m_q}{\gcd(m_q,m)}\right)\right]\\
&= (-1)^q q!  \prod\limits_{i=1}^q \deg(m_i)+\sum\limits_{k=0}^{q-1}{q\choose k} \deg^{q-k}\left(\dfrac{m_q}{\gcd(m_q,m)}\right) \sum\limits_{i=0}^{q-1} (-1)^i \sum\limits_{j=1}^{{q-1}\choose i} \deg^k[\tau_{ij},m]
\end{align*}

Since $\codim(S/M_2) = q-1$,
\[\sum\limits_{i=0}^{q-1} (-1)^i \sum\limits_{j=1}^{{q-1}\choose i} \deg^k [\tau_{ij}] = 0 \text{, for all } 1\leq k\leq q-2.\]
 Since $\codim(S/M_3) = q-1$, 
\[\sum\limits_{i=0}^{q-1}(-1)^i \sum\limits_{j=1}^{{q-1}\choose i} \deg^k[\tau_{ij}] - \sum\limits_{i=0}^{q-1}(-1)^i \sum\limits_{j=1}^{{q-1}\choose i} \deg^k[\tau_{ij},m] = 0 \text{, for all } 1\leq k\leq q-2.\]
Hence,
\[\sum\limits_{i=0}^{q-1}(-1)^i \sum\limits_{j=1}^{{q-1}\choose i} \deg^k[\tau_{ij},m] = 0 \text{, for all } 1\leq k\leq q-2.\]
Also,
\[(-1)^{q-1} (q-1)! \e(S/M_3) = \sum\limits_{i=0}^{q-1} (-1)^i \sum\limits_{j=1}^{{q-1}\choose i} \deg^{q-1} [\tau_{ij}] - \sum\limits_{i=0}^{q-1}(-1)^i \sum\limits_{j=1}^{{q-1}\choose i} \deg^{q-1} [\tau_{ij},m].\]
Thus,
\[\sum\limits_{i=0}^{q-1} (-1)^i \sum\limits_{j=1}^{{q-1}\choose i} \deg^{q-1} [\tau_{ij},m] = \sum\limits_{i=0}^{q-1} (-1)^i \sum\limits_{j=1}^{{q-1}\choose i} \deg^{q-1}[\tau_{ij}] - (-1)^{q-1} (q-1)! \e(S/M_3).\]

Combining these facts, we obtain
\begin{align*}
&(-1)^q q! \e(S/M)\\
& = (-1)^q q! \prod\limits_{i=1}^q \deg(m_i) + \sum\limits_{k=0}^{q-1}  {q\choose k} \deg^{q-k} \left(\dfrac{m_q}{\gcd(m_q,m)}\right) \sum\limits_{i=0}^{q-1} (-1)^i \sum\limits_{j=1}^{{q-1}\choose i} \deg^k [\tau_{ij},m]\\
&= (-1)^q q! \prod\limits_{i=1}^q \deg(m_i) +  \deg^q \left(\dfrac{m_q}{\gcd(m_q,m)}\right) \sum\limits_{i=0}^{q-1}(-1)^i {{q-1}\choose i} +\\
& + q \deg\left(\dfrac{m_q}{\gcd(m_q,m)}\right)\sum\limits_{i=1}^{q-1} (-1)^i \sum\limits_{j=1}^{{q-1}\choose i} \deg^{q-1} [\tau_{ij,m}]\\
&= (-1)^q q! \prod\limits_{i=1}^q \deg(m_i) +\\
&+ q \deg\left(\dfrac{m_q}{\gcd(m_q,m)}\right) \left[ \sum\limits_{i=1}^{q-1} (-1)^i \sum\limits_{j=1}^{{q-1}\choose i} \deg^{q-1} [\tau_{ij}] - (-1)^{q-1} (q-1)! \e(S/M_3)\right]\\
& = (-1)^q q! \prod\limits_{i=1}^q \deg(m_i) + q \deg\left(\dfrac{m_q}{\gcd(m_q,m)}\right) (-1)^{q-1}(q-1)! \e(S/M_2) -\\
& - q \deg\left(\dfrac{m_q}{\gcd(m_q,m)}\right) (-1)^{q-1} (q-1)! \e(S/M_3)\\
&= (-1)^q q! \prod\limits_{i=1}^q \deg(m_i) +\deg\left(\dfrac{m_q}{\gcd(m_q,m)}\right) (-1)^{q-1} q! \prod\limits_{i=1}^{q-1} \deg(m_i) - \\
& - \deg\left(\dfrac{m_q}{\gcd(m_q,m)}\right)(-1)^{q-1}q! \left[\prod\limits_{i=0}^{q-1}\deg(m_i) - \prod\limits_{i=1}^{q-1} \deg\left(\dfrac{m_i}{\gcd(m_i,m)}\right)\right]\\
&= (-1)^q q! \prod\limits_{i=1}^q \deg(m_i) + (-1)^{q-1} q! \prod\limits_{i=1}^q \deg\left(\dfrac{m_i}{\gcd(m_i,m)}\right).
\end{align*}

Hence,
\[\e(S/M) = \prod\limits_{i=1}^q \deg(m_i) - \prod\limits_{i=1}^q \deg\left(\dfrac{m_i}{\gcd(m_i,m)}\right).\]
 \end{proof}

With this new approach, we have been able to prove things without assuming dominance. It is natural to ask whether this reasoning can be applied to extend the results of Sections 5 and 6. For instance, if in the definition of a stem ideal we do not required the ideal to be dominant, would Theorem \ref{Theorem stem-and-leaf} still hold? Moreover, would Theorem \ref{Theorem 2} still hold (in some form) if we did not require the ideal to be dominant?

  \bigskip

\noindent \textbf{Acknowledgements}: I am grateful to Chris Francisco who, after all these years, has remained a mentor to me. My wife Danisa always types my articles and gives me valuable feedback on their content. It is my privilege to work side by side with her. Many thanks, sweetheart!

\end{document}